\documentclass[11pt]{article}
\usepackage[margin=1.02in]{geometry}
\usepackage{amsmath,amssymb,amsthm,mathtools}
\usepackage{microtype}
\usepackage{xcolor}
\definecolor{mygreen}{RGB}{0,100,55}
\definecolor{myred}{RGB}{145,25,25}
\usepackage[colorlinks=true,linkcolor=myred,citecolor=mygreen,urlcolor=mygreen]{hyperref}
\usepackage{enumitem}
\usepackage{booktabs}
\usepackage{tikz}
\usetikzlibrary{arrows.meta,positioning,calc}

\newtheorem{theorem}{Theorem}[section]
\newtheorem{proposition}[theorem]{Proposition}
\newtheorem{lemma}[theorem]{Lemma}
\newtheorem{corollary}[theorem]{Corollary}

\theoremstyle{definition}\newtheorem{definition}[theorem]{Definition}
\theoremstyle{remark}\newtheorem{remark}[theorem]{Remark}\newtheorem{example}[theorem]{Example}
\newcommand{\Z}{\mathbb Z}
\newcommand{\Int}{\operatorname{int}}
\newcommand{\nn}{\langle\!\langle}
\newcommand{\NN}{\rangle\!\rangle}

\title{\textbf{Nonabelian Torus Exteriors in $\mathbb S^4$ and Torus Replacement in Four-Manifolds}}
\author{Anar Akhmedov \and Azer Akhmedov}
\date{}

\begin{document}
\maketitle

\begin{abstract}
We study torus-exterior replacement in four-manifolds, with the construction controlled by the peripheral homomorphism. Using torus exteriors arising from the work of Boyle and Kanenobu-Kazama, with explicitly controlled fundamental groups and peripheral systems, we construct smooth homotopy \(4\)-spheres and manifolds homeomorphic to \(\mathbb S^2\times\mathbb S^2\) and \(\#_3(\mathbb S^2\times\mathbb S^2)\). We also study direct gluings of two torus exteriors. A swap gluing of two one-null exteriors kills both meridians and gives a simply connected manifold with Euler characteristic \(4\) and signature \(0\), while a meridian-preserving gluing of doubly-null exteriors retains the complement groups as an amalgam over the meridian subgroup. Twisted peripheral gluings give cyclic and explicit finite nonabelian fundamental groups. We also consider the full Kanenobu-Kazama \(1\)-handle family and Litherland tori with full-rank peripheral subgroup; the latter have no null primitive boundary slope and lead naturally to amalgamated-product fundamental groups.
\end{abstract}

\section{Introduction}

Let $T\subset \mathbb S^4$ be a smoothly embedded oriented torus and let
\[
 C_T=\mathbb S^4\setminus\operatorname{int}\nu T.
\]
In \cite{AkhPre}, the first author used torus exteriors with $\pi_1(C_T)\cong\mathbb Z$ as replacement pieces in knot-surgery four-manifolds. The fundamental-group calculation, however, does not require the exterior group to be cyclic. What it requires is control of the boundary homomorphism
\[
 \pi_1(\partial C_T)\cong\mathbb Z^3\longrightarrow\pi_1(C_T).
\]
If its image is generated by the geometric meridian, then the kernel is a primitive rank-two summand. Hence there is a boundary basis $(\mu,\alpha,\beta)$ with $\alpha$ and $\beta$ null-homotopic in $C_T$. Once a gluing kills the class identified with $\mu$, normal meridional generation kills the whole, possibly nonabelian, group $\pi_1(C_T)$.

The point of this paper is that this peripheral condition occurs for explicit nonabelian torus complements.  The ingredients have distinct origins.  Boyle classified $1$-handles attached to knotted surfaces in terms of peripheral double cosets and derived the corresponding handle relations \cite{BoyleHandles}.  Kanenobu and Kazama later specialized this framework to $1$-handles on the $6$-twist-spun trefoil and computed both the resulting complement groups and their peripheral subgroups \cite{KK}.  We isolate from their family a subfamily $F_k\subset \mathbb S^4$, $k\ge2$, for which the peripheral image is exactly $\langle\mu_k\rangle\cong\mathbb Z$, whereas the commutator subgroup of the complement group has order $k^3$.  These are the examples used in the present replacement theorem.

Peripheral subgroups of knotted tori had been studied earlier, and the earlier examples should be distinguished from the family used here.  Asano constructed examples for which the nonmeridional part of the peripheral subgroup has rank two \cite{Asano}.  Litherland developed the relation between peripheral data and the second homology of the surface-knot group, and also obtained examples with rank-two nonmeridional peripheral part \cite{Litherland}.  Boyle, in a different construction using attached surface $1$-handles, produced tori whose nonmeridional peripheral part is finite cyclic of orders $2$, $5$, and $10$ \cite{BoyleHandles}; this distinction is summarized explicitly in the introduction of Kanenobu-Kazama \cite{KK}.  These are related peripheral-subgroup phenomena, but they are not the particular family used in our replacement theorem.  Our input is the later Kanenobu-Kazama $G(k,0,m)$ family, for which suitable parameter choices make the nonmeridional peripheral part trivial while the complement group remains nonabelian.

We first prove two elementary absorption lemmas: a two-null version and a one-null version.  We then record the Kanenobu-Kazama presentations and peripheral types.  The small blocks $Y_K$, its rim tori, and the identity-double geometry are taken from the first author's earlier preprint \cite{AkhPre}.  The new input here is to replace the cyclic torus exteriors used there by explicitly marked nonabelian exteriors and to carry the peripheral maps through the fundamental-group and intersection-form calculations.  The full $1$-handle family is used in the $Y_K$ block with the one-null gluing; the trivial-type subfamily is kept separate and is also used in the identity-double construction. The resulting closed manifolds are simply connected with intersection forms $H$ and $3H$, hence are homeomorphic to $\mathbb S^2\times\mathbb S^2$ and $\#_3(\mathbb S^2\times\mathbb S^2)$, respectively. We keep the group calculation in pushout form throughout; in particular, the nonabelian exterior groups are not discarded until their meridians have been proved trivial.

We also glue two torus exteriors directly.  A swap gluing of two one-null
exteriors identifies each meridian with a null slope on the opposite side
and gives a simply connected manifold with Euler characteristic $4$ and
signature $0$.  For doubly-null exteriors, a meridian-preserving gluing
retains the two complement groups as an amalgam over the meridian subgroup.
These constructions apply to the full Kanenobu-Kazama $1$-handle family.
Twisted peripheral gluings give cyclic groups and explicit finite
nonabelian groups.  At the other end of the peripheral range, Litherland
tori have full-rank peripheral subgroup and no null primitive boundary
slope; their replacement groups are therefore kept as amalgamated-product
presentations.

\section{Normally marked torus exteriors}\label{sec:general}

Let $T\subset \mathbb S^4$ be an oriented smoothly embedded torus and set
\[
 C_T=\mathbb S^4\setminus \Int\nu T,\qquad \partial C_T\cong \mathbb T^3.
\]
Alexander duality gives
\[
 H_1(C_T;\Z)\cong\Z,\qquad H_2(C_T;\Z)\cong\Z^2,
\]
and therefore $e(C_T)=2$.  Novikov additivity in
\(\mathbb S^4=C_T\cup(\mathbb T^2\times\mathbb D^2)\) gives $\sigma(C_T)=0$.

\begin{lemma}\label{lem:normal-meridian}
For every smoothly embedded oriented torus $T\subset \mathbb S^4$, a meridian $\mu_T$ normally generates $\pi_1(C_T)$.
\end{lemma}
\begin{proof}
Write \(G=\pi_1(C_T)\).  Reattaching the tubular neighborhood
\[
\nu T\cong\mathbb T^2\times\mathbb D^2
\]
by the original boundary identification recovers \(\mathbb S^4\).  Choose
the usual boundary basis \((a,b,\mu_T)\), where \(a,b\) are the two
surface directions and \(\mu_T\) is the boundary of a normal disk.  Under
the inclusion into \(\mathbb T^2\times\mathbb D^2\), the classes \(a,b\)
map to the two generators of
\(\pi_1(\mathbb T^2\times\mathbb D^2)\cong\mathbb Z^2\), while
\(\mu_T\) maps to the identity.

Seifert-van Kampen therefore gives the pushout
\[
G *_{\pi_1(\mathbb T^3)}
\pi_1(\mathbb T^2\times\mathbb D^2).
\]
The two generators coming from the tubular neighborhood are already
identified with the images of \(a,b\) in \(G\), so they do not contribute
new generators to the pushout.  The only new relation imposed on \(G\) is
that the meridian becomes trivial.  Thus
\[
\pi_1(\mathbb S^4)
\cong G/\nn\mu_T\NN.
\]
Since \(\pi_1(\mathbb S^4)=1\), the normal closure of \(\mu_T\) is all of
\(G\).  Hence \(\mu_T\) normally generates \(\pi_1(C_T)\).
\end{proof}

\begin{definition}\label{def:DN}
A \emph{doubly-null marked torus exterior} is a quadruple
\[
 (C_T;\mu,\alpha,\beta)
\]
for which $(\mu,\alpha,\beta)$ is a primitive basis of $H_1(\partial C_T;\Z)$, $\mu$ is a meridian of $T$, and
\[
 \alpha=\beta=1\quad\text{in }\pi_1(C_T).
\]
No assumption is made that $\pi_1(C_T)$ is abelian or cyclic.
\end{definition}

\begin{lemma}\label{lem:cyclic-peripheral-kernel}
Suppose that the boundary inclusion
\[
 i_*:\pi_1(\partial C_T)\cong\mathbb Z^3\longrightarrow \pi_1(C_T)
\]
has image exactly the infinite cyclic subgroup generated by the geometric meridian $\mu$.  Then there are primitive classes $\alpha,\beta\in\pi_1(\partial C_T)$ such that $(\mu,\alpha,\beta)$ is a basis of $\mathbb Z^3$ and
\[
 i_*(\alpha)=i_*(\beta)=1.
\]
In particular, $\alpha$ and $\beta$ are represented by boundary loops that are null-homotopic in $C_T$.
\end{lemma}
\begin{proof}
Identify \(\operatorname{im}i_*\) with \(\mathbb Z\) by sending
\(i_*(\mu)\) to \(1\).  We obtain a surjective homomorphism
\[
q:\mathbb Z^3\longrightarrow\mathbb Z
\]
with \(q(\mu)=1\).  The map
\[
s:\mathbb Z\longrightarrow\mathbb Z^3,\qquad s(1)=\mu,
\]
is a section of \(q\).  Hence
\[
\mathbb Z^3=\mathbb Z\langle\mu\rangle\oplus\ker q.
\]
In particular, \(\ker q\) is free of rank two and is primitive: the
quotient \(\mathbb Z^3/\ker q\cong\mathbb Z\) is torsion free.

Choose an integral basis \(\alpha,\beta\) of \(\ker q\).  The direct-sum
decomposition shows that
\[
(\mu,\alpha,\beta)
\]
is a basis of \(\mathbb Z^3\).  Since \(\alpha,\beta\in\ker q\), their
images under \(i_*\) are the identity element of
\(\operatorname{im}i_*\), and hence of \(\pi_1(C_T)\).  Thus the
corresponding boundary loops are null-homotopic in \(C_T\).
\end{proof}

\begin{lemma}\label{lem:T3gluing}
Let $(u_1,u_2,u_3)$ and $(v_1,v_2,v_3)$ be primitive bases of $H_1(\mathbb T^3;\mathbb Z)$.  After replacing one $v_i$ by its inverse if necessary, there is an orientation-reversing diffeomorphism of $\mathbb T^3$ carrying $u_i$ to $v_i$ for $i=1,2,3$.
\end{lemma}
\begin{proof}
Use the two primitive bases to identify
\[
H_1(\mathbb T^3;\mathbb Z)\cong\mathbb Z^3.
\]
The correspondence \(u_i\mapsto v_i\) is then represented by a matrix
\(A\in\mathrm{GL}(3,\mathbb Z)\).  Regard \(A\) as an integral linear map
of \(\mathbb R^3\).  Since \(A(\mathbb Z^3)=\mathbb Z^3\), it descends to
a linear diffeomorphism
\[
f_A:\mathbb R^3/\mathbb Z^3\longrightarrow
\mathbb R^3/\mathbb Z^3
\]
whose action on \(H_1\) is exactly \(A\).

The orientation of this diffeomorphism is the sign of \(\det A\).  For
gluing oriented four-manifolds along their boundaries we need an
orientation-reversing boundary map.  If \(\det A=-1\), \(f_A\) already
has the required orientation.  If \(\det A=+1\), replace one target basis
element, say \(v_3\), by \(v_3^{-1}\).  This composes the prescribed
homology map with a diagonal matrix of determinant \(-1\), and hence
changes the determinant to \(-1\).  Reversing the orientation of a
boundary loop does not change whether that loop is null-homotopic, nor
does it change the cyclic subgroup or normal closure generated by it.
Thus the modified basis correspondence has all the required properties
and is realized by an orientation-reversing diffeomorphism of
\(\mathbb T^3\).
\end{proof}

\begin{lemma}\label{lem:replacement}
Let $M$ contain pairwise disjoint square-zero tori $R_1,\dots,R_r$, put
\[
 W=M\setminus \Int(\nu R_1\cup\cdots\cup\nu R_r),
\]
and choose on $\partial\nu R_i$ a primitive basis $(\lambda_i,\eta_i,\rho_i)$, where $\rho_i$ is the meridian of $R_i$.  Let $(C_i;\mu_i,\alpha_i,\beta_i)$ be doubly-null marked torus exteriors and glue by
\[
 \lambda_i\mapsto\alpha_i,\qquad
 \rho_i\mapsto\beta_i,\qquad
 \eta_i\mapsto\mu_i.
\]
If
\[
 \pi_1(M)/\nn\lambda_1,\dots,\lambda_r\NN=1,
\]
then the resulting closed manifold is simply connected.
\end{lemma}
\begin{proof}
Write
\[
G_W=\pi_1(W),\qquad G_i=\pi_1(C_i).
\]
For each boundary component, van Kampen identifies the three boundary
classes in \(G_W\) with their prescribed images in \(G_i\).  Thus the
fundamental group of the glued manifold is obtained from
\[
G_W*G_1*\cdots *G_r
\]
by imposing
\[
\lambda_i=\alpha_i,\qquad
\rho_i=\beta_i,\qquad
\eta_i=\mu_i
\qquad (1\leq i\leq r).
\]
Because the marking on \(C_i\) is doubly null,
\[
\alpha_i=\beta_i=1\quad\text{in }G_i.
\]
Consequently the first two boundary relations become
\[
\lambda_i=1,\qquad \rho_i=1,
\]
while the third relation remains
\[
\eta_i=\mu_i.
\]

It is useful to compare this quotient with the original filling of \(W\).
Restoring \(\nu R_i\cong\mathbb T^2\times\mathbb D^2\) kills the meridian
\(\rho_i\) and recovers \(M\).  Hence
\[
G_W/\nn\rho_1,\dots,\rho_r\NN\cong\pi_1(M).
\]
After the replacement gluing we have already imposed all the relations
\(\rho_i=1\), so the image of \(G_W\) first factors through
\(\pi_1(M)\).  We also impose \(\lambda_i=1\).  Therefore the image of
\(G_W\) in the final group factors through
\[
\pi_1(M)/\nn\lambda_1,\dots,\lambda_r\NN.
\]
By hypothesis this quotient is trivial.  Thus every element coming from
\(G_W\), and in particular each \(\eta_i\), is trivial in the final
group.

The remaining boundary relation \(\eta_i=\mu_i\) now gives
\[
\mu_i=1
\]
for every inserted exterior.  By Lemma~\ref{lem:normal-meridian}, the
geometric meridian \(\mu_i\) normally generates \(G_i\).  Hence
\[
G_i/\nn\mu_i\NN=1,
\]
so every exterior group \(G_i\) also dies.  All factors in the van Kampen
pushout are therefore trivial, and the glued manifold is simply
connected.
\end{proof}

\begin{lemma}\label{lem:one-null}
Let $M$ contain pairwise disjoint square-zero tori $R_1,\dots,R_r$ and put
\[
W=M\setminus\operatorname{Int}(\nu R_1\cup\cdots\cup\nu R_r).
\]
On $\partial\nu R_i$ choose a primitive basis $(\lambda_i,\eta_i,\rho_i)$, with $\rho_i$ the meridian.  Let $C_i$ be torus exteriors with a primitive boundary basis $(\mu_i,u_i,v_i)$ such that $\mu_i$ is the geometric meridian and
\[
v_i=1\quad\text{in }\pi_1(C_i).
\]
Glue by
\[
\lambda_i\longmapsto u_i,\qquad
\eta_i\longmapsto\mu_i,\qquad
\rho_i\longmapsto v_i.
\]
Assume that, after the original fillings $W\to M$, the classes $\eta_i$ are trivial in $\pi_1(M)$ and that
\[
\pi_1(M)/\nn\lambda_1,\dots,\lambda_r\NN=1.
\]
Then the resulting closed manifold is simply connected.
\end{lemma}
\begin{proof}
Write \(G_W=\pi_1(W)\) and \(G_i=\pi_1(C_i)\).  Van Kampen gives a
quotient of
\[
G_W*G_1*\cdots*G_r
\]
by the relations
\[
\lambda_i=u_i,\qquad
\eta_i=\mu_i,\qquad
\rho_i=v_i.
\]
Since \(v_i=1\) in \(G_i\), the last relation gives
\[
\rho_i=1.
\]
Imposing all the meridional relations \(\rho_i=1\) on \(G_W\) is exactly
what occurs when the deleted neighborhoods \(\nu R_i\) are restored.
Therefore the image of \(G_W\) now factors through
\[
G_W/\nn\rho_1,\dots,\rho_r\NN\cong\pi_1(M).
\]

By assumption, the image of each \(\eta_i\) is trivial in this original
filling quotient.  Hence \(\eta_i=1\) in the replacement group as well.
The relation \(\eta_i=\mu_i\) then gives
\[
\mu_i=1.
\]
Lemma~\ref{lem:normal-meridian} says that \(\mu_i\) normally generates
\(G_i\), so the entire inserted exterior group \(G_i\) becomes trivial.
In particular,
\[
u_i=1.
\]
The remaining relation \(\lambda_i=u_i\) therefore gives
\[
\lambda_i=1
\]
for every \(i\).

At this point all inserted exterior groups have disappeared, and the
remaining image of \(G_W\) is a quotient of
\[
\pi_1(M)/\nn\lambda_1,\dots,\lambda_r\NN.
\]
This group is trivial by hypothesis.  Hence no generator from either the
central piece or any inserted exterior survives, and the resulting
manifold is simply connected.
\end{proof}

\section{Explicit nonabelian admissible exteriors from Boyle $1$-handles}\label{sec:KK}

We need torus exteriors for which both the nonabelian group and the peripheral map are controlled.  Boyle's work provides the classification of attached surface $1$-handles and the associated handle relations \cite{BoyleHandles}.  Kanenobu and Kazama use that framework for the $6$-twist-spun trefoil and perform the explicit complement-group and peripheral-subgroup calculations needed here \cite{KK}.

\subsection{Why these are genuine $\mathbb T^3$-boundary replacement pieces}
We first record a basic point about the ambient embedding.  The Kanenobu-Kazama objects used below are not merely complements of $2$-knots.  They begin with a $2$-knot $K\cong \mathbb S^2\subset \mathbb S^4$ and attach a surface $1$-handle $h$.  The resulting connected oriented surface
\[
 F=K+h
\]
has genus one; hence $F\cong \mathbb T^2$ is a smoothly embedded oriented torus in $\mathbb S^4$.

\begin{lemma}\label{lem:KK-square-zero}
Every smoothly embedded oriented torus $F\subset \mathbb S^4$ has self-intersection zero.  Consequently its oriented normal bundle is trivial, so
\[
 F^2=0,\qquad \nu F\cong \mathbb T^2\times\mathbb D^2,
 \qquad \partial\nu F\cong \mathbb T^3.
\]
In particular, every Kanenobu-Kazama torus used below is a square-zero torus and its exterior is a legitimate $\mathbb T^3$-boundary replacement piece.
\end{lemma}

\begin{proof}
Since $H_2(\mathbb S^4;\mathbb Z)=0$, the homology class $[F]$ vanishes.  Hence
\[
 F^2=[F]\cdot[F]=0.
\]
For an oriented embedded surface in an oriented $4$-manifold, the self-intersection number is the Euler number of its oriented rank-two normal bundle $\nu F$.  Thus $e(\nu F)=0\in H^2(\mathbb T^2;\mathbb Z)$.  Oriented real rank-two bundles over $\mathbb T^2$ are classified by their Euler class, so $\nu F$ is trivial.  Therefore $\nu F\cong \mathbb T^2\times\mathbb D^2$ and $\partial\nu F\cong \mathbb T^2\times\mathbb S^1=\mathbb T^3$.
\end{proof}

By Lemma~\ref{lem:KK-square-zero}, its tubular neighborhood is
\[
 \nu F\cong \mathbb T^2\times\mathbb D^2,
\]
and its exterior
\[
 C_F=\mathbb S^4\setminus\operatorname{Int}\nu F
\]
has
\[
 \mathord{\partial C_F\cong \mathbb T^2\times\mathbb S^1=\mathbb T^3.}
\]
Equivalently,
\[
 \mathbb S^4=C_F\cup_{\mathbb T^3}(\mathbb T^2\times\mathbb D^2).
\]
Thus $C_F$ has exactly the boundary required to replace a deleted square-zero torus neighborhood $\nu R\cong \mathbb T^2\times\mathbb D^2$ in a closed $4$-manifold.  In particular, all gluings below are genuine diffeomorphisms
\[
 \phi:\partial C_F\longrightarrow\partial\nu R\cong \mathbb T^3,
\]
not gluings along the $\mathbb S^2\times\mathbb S^1$ boundary of the original $2$-knot exterior.

The three boundary directions will be denoted $(\mu,\alpha,\beta)$, with $\mu$ the geometric meridian of the torus $F$.  For the special Kanenobu-Kazama subfamily used below, their peripheral calculation implies that the image of
\[
 \iota_*:\pi_1(\partial C_F)\cong\mathbb Z^3\longrightarrow\pi_1(C_F)
\]
is precisely the cyclic subgroup generated by $\mu$.  By Lemma~\ref{lem:cyclic-peripheral-kernel}, the kernel is a primitive rank-two direct summand.  Choosing a basis $\alpha,\beta$ of this kernel gives
\[
 \alpha=\beta=1\quad\text{in }\pi_1(C_F).
\]
This is the precise $\mathbb T^3$ boundary information used in the replacement gluing.  The fact that $\pi_1(C_F)$ itself is nonabelian is compatible with this: the exterior group is normally generated by the meridian, while the two selected kernel slopes die.  We do not identify these slopes with the two obvious push-offs of a chosen basis of $H_1(F)$ unless that geometric identification has been computed separately.

For a knotted torus $F\subset \mathbb S^4$, write
\[
 C_F=\mathbb S^4\setminus\operatorname{Int}\nu F,
 \qquad G_F=\pi_1(C_F),
\]
and let
\[
 P_F=\operatorname{im}\bigl(\pi_1(\partial C_F)\longrightarrow G_F\bigr)
\]
be its peripheral subgroup.  Kanenobu-Kazama recall that
\[
 P_F\cong \mathbb Z\langle\mu\rangle\oplus\tau_F,
\]
where $\mu$ is a meridian and $\tau_F$ is a quotient of $\pi_1(F)\cong\mathbb Z^2$; they call $\tau_F$ the \emph{type} of $F$ \cite[Section~1]{KK}.  Thus $\tau_F=0$ is exactly the peripheral condition needed here: the image of the entire boundary group is generated by the meridian alone.

Let $K\subset \mathbb S^4$ be the $6$-twist-spun trefoil.  Its group is denoted $\pi$ in \cite{KK}.  For a $1$-handle $h$ attached to $K$, Boyle associates an orbit $T(h)$ in the commutator subgroup of $\pi$.  Kanenobu-Kazama use Boyle's calculation to obtain
\[
 \pi_1(\mathbb S^4\setminus\nu(K+h))\cong \pi/[t,T(h)],
\]
and show that the peripheral subgroup is generated by the meridian $t$ together with $T(h)$ \cite[Proposition~1]{KK}.

\begin{proposition}\label{prop:KKpresentation}
Kanenobu-Kazama record the following presentation for the $6$-twist-spun trefoil group, citing Zeeman's twist-spinning work \cite{Zeeman,KK}:
\[
 \pi=\langle x,y\mid xyx=yxy,\ [x^6,y]=1\rangle,
\]
where $x$ and $y$ are meridians.  We take $t=x$.  If $h$ is a $1$-handle with Boyle orbit $T(h)\subset\pi'$, then
\[
 \pi_1\bigl(\mathbb S^4\setminus\nu(K+h)\bigr)
 \cong
 \pi/\nn [x,u]\; ;\;u\in T(h)\NN .
\]
Equivalently, after choosing words $u_1,\dots,u_s$ in $x,y$ representing the orbit,
\[
 \pi_1\bigl(\mathbb S^4\setminus\nu(K+h)\bigr)
 \cong
 \left\langle x,y\ \middle|\ xyx=yxy,\ [x^6,y]=1,
 [x,u_1]=\cdots=[x,u_s]=1\right\rangle .
\]
For the handles $h_w$ used by Kanenobu-Kazama with nontrivial orbit of length six,
\[
 T(h_w)=\{w,xwx^{-1},\ldots,x^5wx^{-5}\},
\]
so (2.2) becomes
\[
 G_w=\left\langle x,y\ \middle|\ xyx=yxy,\ [x^6,y]=1,
 [x,x^nwx^{-n}]=1\ (0\le n\le5)\right\rangle .
\]
Here $w$ is the specific element of $\pi'$ chosen in the Kanenobu-Kazama normal form.  We will use (2.1)-(2.3), rather than an expanded word in $x,y$, whenever no expanded word is needed; this avoids introducing an unverified change of generators.
\end{proposition}
\begin{proof}
Boyle's handle calculation is the source of the quotient relation; Kanenobu-Kazama package it as their Proposition~1 for the present setting \cite{BoyleHandles,KK}.  Thus
$\pi(K+h)\cong\pi/[t,T(h)]$, where $[t,T(h)]$ denotes the normal closure of all commutators $[t,u]$ with $u\in T(h)$.  Substituting $t=x$ gives (2.1).  Their description of the orbit of $h_w$ gives (2.3).
\end{proof}

Kanenobu-Kazama consider handles $h_w$ indexed by elements $w=a^k b^\ell c^m$ in the commutator subgroup and denote the resulting complement group by $G(k,\ell,m)$.  For the subfamily $\ell=0$, the orbit representative has the form $w=a^k c^m$.  Their Claims~6 and~7 show that, for $k\ne0$,
\[
|G'(k,0,m)|=|k|^3,
\]
and compute the order of the peripheral orbit element: for odd $k>0$ it is $k/\gcd(k,m)$, whereas for even $k=2r>0$ it is $2r/\gcd(2r,m+r)$. These are the only additional group-theoretic facts from \cite{KK} used below.

\begin{theorem}\label{thm:KK-family}
For every integer $k\ge2$ there is a smoothly embedded torus $F_k\subset \mathbb S^4$, obtained by attaching a Boyle $1$-handle to the $6$-twist-spun trefoil, such that
\begin{enumerate}[label=(\roman*)]
\item $P_{F_k}=\mathbb Z\langle\mu_k\rangle$ is generated by a meridian;
\item $G_k:=\pi_1(\mathbb S^4\setminus\nu F_k)$ is nonabelian;
\item the commutator subgroup $G_k'$ has order $k^3$.
\end{enumerate}
Consequently $C_{F_k}$ admits a doubly-null marking, and the groups $G_k$ are pairwise nonisomorphic.
\end{theorem}
\begin{proof}
Use the Kanenobu-Kazama family $G(k,0,m)$.  If $k$ is odd, take $m=k$; then the orbit element has order $k/\gcd(k,k)=1$.  If $k=2r$ is even, take $m=r$; then its order is $2r/\gcd(2r,2r)=1$.  Proposition~1 of \cite{KK} therefore gives $\tau_{F_k}=0$, hence $P_{F_k}=\mathbb Z\langle\mu_k\rangle$.  Claims~6 and~7 of \cite{KK} give $|G_k'|=k^3$.  Since $k\ge2$, the commutator subgroup is nontrivial, so $G_k$ is nonabelian.  Since the peripheral image is exactly $\langle\mu_k\rangle$, and $[\mu_k]$ generates
$H_1(C_{F_k};\mathbb Z)\cong\mathbb Z$, the meridian has infinite order.  Hence the boundary map fits into a split exact sequence
\[
0\longrightarrow K_k\longrightarrow \pi_1(\partial C_{F_k})\cong\mathbb Z^3
\longrightarrow \langle\mu_k\rangle\cong\mathbb Z\longrightarrow0,
\]
where the splitting is given by the geometric meridian.  Thus $K_k\cong\mathbb Z^2$ is a primitive direct summand.  Choose a primitive basis $\alpha_k,\beta_k$ of $K_k$.  By definition of the kernel,
\[
\alpha_k=\beta_k=1\quad\text{in }G_k,
\]
and $(\mu_k,\alpha_k,\beta_k)$ is a primitive basis of the boundary group.  Therefore $(C_{F_k};\mu_k,\alpha_k,\beta_k)$ is doubly-null.  Finally, the orders $|G_k'|=k^3$ distinguish the groups.
\end{proof}

\begin{example}[the first member]\label{ex:KK2}
For $k=2$ take $m=1$.  Kanenobu-Kazama note that the commutator subgroup $G'(2,0,m)$ is the quaternion group of order eight.  The chosen orbit has order one, so the associated torus has peripheral subgroup $\mathbb Z\langle\mu\rangle$ but nonabelian complement group.  This is the smallest explicit example in the family above.
\end{example}

\section{The full Kanenobu-Kazama $1$-handle family}\label{sec:fullKK}
We retain the full $1$-handle family $G(k,\ell,m)$ from Section~2 of \cite{KK}, without imposing the type-zero condition.  The later satellite constructions in \cite{KK} are not used here.

Let
\[
w=a^k b^\ell c^m\in\pi'
\]
and let $F(k,\ell,m)=K+h_w$ be the torus obtained from the $6$-twist-spun trefoil $K$ by the corresponding $1$-handle.  Write
\[
G(k,\ell,m)=\pi_1\bigl(\mathbb S^4\setminus\nu F(k,\ell,m)\bigr).
\]
Proposition~1 of \cite{KK} gives
\[
G(k,\ell,m)\cong \pi/[x,T(h_w)],
\]
and says that the peripheral subgroup is generated by the meridian $x$ and the orbit $T(h_w)$.  Consequently the extra peripheral factor is cyclic.  Let $p=p(k,\ell,m)$ denote the order of the orbit element in $G(k,\ell,m)$, with $p=0$ denoting infinite order as in \cite{KK}.  To avoid using $\mathbb Z_0$ for an infinite cyclic group, we write
\[
P_{F(k,\ell,m)}\cong
\begin{cases}
\mathbb Z\langle\mu\rangle\oplus\mathbb Z_p,& p>0,\\
\mathbb Z\langle\mu\rangle\oplus\mathbb Z,& p=0.
\end{cases}
\]

\begin{lemma}\label{lem:cyclic-type-one-null}
For every torus in the Kanenobu-Kazama $1$-handle family there is a primitive basis
\[
(\mu,u,v)
\]
of $\pi_1(\partial C_F)\cong\mathbb Z^3$ such that $\mu$ is the geometric meridian,
\[
v=1\quad\text{in }\pi_1(C_F),
\]
and the image of $u$ generates the cyclic type factor $\tau_F$.  If $p=1$, both $u$ and $v$ may be chosen null-homotopic.
\end{lemma}
\begin{proof}
Let
\[
i_*:H_1(\partial C_F;\mathbb Z)\cong\mathbb Z^3\longrightarrow P_F
\]
be induced by the boundary inclusion.  The meridian has infinite order and generates the meridional direct summand of $P_F$.  Choose a primitive rank-two summand $L\cong\mathbb Z^2$ complementary to $\mathbb Z\langle\mu\rangle$ in the boundary group.  The restriction
\[
i_*|_L:L\longrightarrow\tau_F
\]
has image equal to the cyclic type factor.  If $p>0$, this is a surjection $\mathbb Z^2\to\mathbb Z_p$; if $p=0$, it is a surjection $\mathbb Z^2\to\mathbb Z$.

In the finite case, Smith normal form gives a primitive basis $(u,v)$ of $L$ for which $i_*(u)$ generates the cyclic image and $i_*(v)=1$.  In the infinite case, the epimorphism $\mathbb Z^2\to\mathbb Z$ splits; its kernel is generated by a primitive element $v$, which can be completed by a lift $u$ of a generator of $\mathbb Z$.  Thus in both cases
\[
(\mu,u,v)
\]
is a primitive basis of the boundary group, $i_*(u)$ generates the type factor, and $v$ is null-homotopic in $C_F$.  If $p=1$, the type factor is trivial, so the entire summand $L$ lies in the kernel and both $u$ and $v$ may be chosen null-homotopic.
\end{proof}

\begin{proposition}\label{prop:KK-group-cases}
The following statements are consequences of Section~2 of \cite{KK}.
\begin{enumerate}[label=(\roman*)]
\item For $k=\ell=0$, one has
\[
G(0,0,m)=\pi=\langle x,y\mid xyx=yxy,\ [x^6,y]=1\rangle.
\]
The orbit of $c^m$ has length one.  For $m\ne0$ the extra peripheral factor is infinite cyclic, while for $m=0$ it is trivial.
\item If $(k,\ell)\ne(0,0)$, then the commutator subgroup $G'(k,\ell,m)$ is finite (Claim~5 of \cite{KK}).  Since every torus complement in $\mathbb S^4$ has abelianization $\mathbb Z$, the full group is finite-by-$\mathbb Z$.  It is cyclic precisely when the commutator subgroup is trivial; otherwise it is nonabelian.
\item For the subfamily $G(k,0,m)$ with $k>0$, Claims~6 and~7 of \cite{KK} give
\[
|G'(k,0,m)|=k^3.
\]
Thus $G(1,0,m)\cong\mathbb Z$, whereas $G(k,0,m)$ is nonabelian for every $k\ge2$.
\item If $k$ is odd, the type order in $G(k,0,m)$ is
\[
p=\frac{k}{\gcd(k,m)}.
\]
If $k=2r$ is even, it is
\[
p=\frac{2r}{\gcd(2r,m+r)}
 =\frac{k}{\gcd(k,m+k/2)}.
\]
Hence the type-zero cases are exactly $k\mid m$ for odd $k$, and $m\equiv k/2\pmod{k}$ for even $k$.
\end{enumerate}
\end{proposition}
\begin{proof}
Part (i) is Claim~3 and the discussion immediately following it in \cite{KK}: $G(0,0,m)=\pi$, and the center of $\pi'$ is infinite cyclic generated by $c$.  Part (ii) is Claim~5, together with Alexander duality $H_1(C_F;\mathbb Z)\cong\mathbb Z$.  Parts (iii) and (iv) are Claims~6 and~7.  In particular, for $k=1$ the commutator subgroup has order one, so the whole group equals its abelianization $\mathbb Z$.
\end{proof}

Two small examples illustrate the distinction between complement group and peripheral type.  For $G(2,0,m)$ the commutator subgroup is the quaternion group $Q_8$ \cite[Remark after Claim~7]{KK}.  The type has order
\[
p=\frac{2}{\gcd(2,m+1)}.
\]
Thus odd $m$ gives the doubly-null case $p=1$, whereas even $m$ gives type $\mathbb Z_2$; the complement group remains nonabelian in both cases.  For $G(3,0,m)$ the commutator subgroup has order $27$, while the type is trivial when $3\mid m$ and is $\mathbb Z_3$ otherwise.

\section{Litherland tori with full peripheral rank}\label{sec:litherland}

Litherland constructed knotted tori in \(\mathbb S^4\) for which the nonmeridional part of the peripheral subgroup has rank two \cite{Litherland}. In the notation used in \cite{KK}, these are tori of type
\[
\tau_F\cong \mathbb Z\oplus\mathbb Z .
\]
Accordingly,
\[
P_F=\operatorname{im}\{\pi_1(\partial C_F)\to\pi_1(C_F)\}
\cong \mathbb Z\langle\mu\rangle\oplus\mathbb Z\oplus\mathbb Z .
\]
Kanenobu-Kazama explicitly distinguish these examples from the finite-cyclic types obtained from attached \(1\)-handles \cite[Introduction]{KK}.

\begin{proposition}\label{prop:litherland-injective}
For a Litherland torus \(F\subset \mathbb S^4\) of type \(\mathbb Z\oplus\mathbb Z\), the boundary homomorphism
\[
i_*:\pi_1(\partial C_F)\cong\mathbb Z^3\longrightarrow \pi_1(C_F)
\]
is injective. In particular, no nontrivial primitive boundary slope is null-homotopic in \(C_F\).
\end{proposition}

\begin{proof}
By definition of the type and peripheral subgroup,
\[
\operatorname{im}i_*\cong
\mathbb Z\langle\mu\rangle\oplus\mathbb Z\oplus\mathbb Z
\cong\mathbb Z^3.
\]
Thus \(i_*\) is a surjection from \(\mathbb Z^3\) onto a free abelian group of the same rank. Its kernel has rank zero. Since every subgroup of \(\mathbb Z^3\) is torsion-free, the kernel is trivial.
\end{proof}

Thus these examples lie at the opposite peripheral extreme from the doubly-null Kanenobu-Kazama family. They do not satisfy the hypotheses of Lemma~\ref{lem:replacement} or Lemma~\ref{lem:one-null}.

Choose a primitive basis
\[
(\mu,u,v)
\]
of \(\pi_1(\partial C_F)\), and write
\[
G_F=\pi_1(C_F),\qquad
\bar\mu=i_*(\mu),\quad \bar u=i_*(u),\quad \bar v=i_*(v).
\]
All three peripheral elements are nontrivial and generate
\[
P_F\cong\mathbb Z^3.
\]

Let \(R\subset M\) be a square-zero torus, let
\[
W=M\setminus\operatorname{Int}\nu R,
\]
and choose a boundary basis \((\lambda,\eta,\rho)\) on \(\partial\nu R\). Glue \(C_F\) by
\[
\lambda\longmapsto u,\qquad
\eta\longmapsto\mu,\qquad
\rho\longmapsto v.
\]
Seifert-van Kampen gives
\[
\pi_1(W\cup_{\partial} C_F)
\cong
\frac{\pi_1(W)*G_F}
{\langle\!\langle
\lambda=\bar u,\;
\eta=\bar\mu,\;
\rho=\bar v
\rangle\!\rangle}.
\tag{\(\dagger\)}
\]
Unlike the one-null and two-null cases, none of the three boundary relations is automatically a killing relation in \(G_F\).

For the two-rim-torus construction in \(Y_K\), using two Litherland exteriors \(C_{F_1},C_{F_2}\) with peripheral bases
\[
(\mu_i,u_i,v_i),\qquad i=1,2,
\]
the corresponding group is
\[
\pi_1(Z_L)=
\frac{
\pi_1(W)*G_{F_1}*G_{F_2}}
{\left\langle\!\left\langle
y=\bar u_1,\;
\mu_y=\bar\mu_1,\;
\rho_y=\bar v_1,\;
d=\bar u_2,\;
\mu_d=\bar\mu_2,\;
\rho_d=\bar v_2
\right\rangle\!\right\rangle},
\]
where
\[
W=Y_K\setminus\operatorname{Int}(\nu R_y\cup\nu R_d).
\]
Without an additional marked presentation for the Litherland complement groups, this van Kampen pushout presentation is the correct conclusion. In contrast with the Kanenobu-Kazama type-zero and finite-cyclic peripheral cases, simple connectivity is not expected here in general: the full-rank peripheral map has no null primitive boundary slope, so the purpose of these replacements is to retain and compute the resulting nontrivial fundamental group rather than to force it to vanish.

The three peripheral regimes used in this paper may be summarized as
\[
\begin{array}{c|c|c}
\text{family} & \tau_F & \text{primitive null boundary slopes}\\
\hline
\text{Kanenobu-Kazama type }0 & 0 & 2\\
\text{Kanenobu-Kazama type }\mathbb Z_p,\ p>1 & \mathbb Z_p & 1\\
\text{Litherland full-rank examples} & \mathbb Z\oplus\mathbb Z & 0.
\end{array}
\]

\section{Applications to the small knot-surgery blocks}\label{sec:small}

We apply Lemma~\ref{lem:replacement} to the trefoil block of \cite{AkhPre}.  Here \(M_K\) denotes zero surgery on the trefoil \(K\).  Each copy of \(M_K\times\mathbb S^1\) contains the square-zero torus fiber \(F\) and a square-zero torus section \(T_m\), and
\[
Y_K=(M_K\times \mathbb S^1)\#_{F=T_m}(M_K\times \mathbb S^1).
\]
The unused section in the first summand and unused fiber in the second summand glue to the square-zero genus-two surface \(\Sigma_2\subset Y_K\).  We use below only the displayed presentation of \(\pi_1(Y_K)\), the specified peripheral classes of the rim tori, and the geometric representatives of the hyperbolic pair.  These data are recorded here so that the fundamental-group calculations do not require the reader to reconstruct them from \cite{AkhPre}.

The presentation is
\[
\begin{aligned}
\pi_1(Y_K)=\langle a,b,x,d,y\mid {}&aba=bab,\ [x,a]=[x,b]=1,\\
&[y,x]=[y,b]=1,\ dxd^{-1}=xb,\\
&dbd^{-1}=x^{-1},\ ab^2ab^{-4}=[d,y]\rangle.
\end{aligned}
\]
As in \cite{AkhPre}, one immediately checks
\[
 \pi_1(Y_K)/\nn d,y\NN=1.\tag{\(*\)}
\]
For completeness, the collapse is explicit.  Setting $d=y=1$ in the displayed presentation, the relation $dxd^{-1}=xb$ gives $b=1$; then $dbd^{-1}=x^{-1}$ gives $x=1$; finally $aba=bab$ with $b=1$ gives $a=1$.  Hence no generator survives.

\subsection{The full Kanenobu-Kazama family: a one-null gluing}\label{subsec:fullKKapp}
The full $1$-handle family can be used in the trefoil block even when the type factor is nontrivial.  The gluing is different from the doubly-null gluing used later.

For each inserted exterior choose the basis $(\mu_i,u_i,v_i)$ of Lemma~\ref{lem:cyclic-type-one-null}, so $v_i$ is null-homotopic and $u_i$ records the cyclic type.  On the two rim-torus boundaries use the bases
\[
(y,\mu_y,\rho_y),\qquad(d,\mu_d,\rho_d).
\]
Choose orientation-reversing boundary diffeomorphisms with
\[
y\longmapsto u_1,\qquad \mu_y\longmapsto\mu_1,\qquad \rho_y\longmapsto v_1,
\]
\[
d\longmapsto u_2,\qquad \mu_d\longmapsto\mu_2,\qquad \rho_d\longmapsto v_2.
\]
Let the resulting closed manifold be denoted by
\[
Z_K^{\mathrm{full}}(C_1,C_2).
\]

\begin{theorem}\label{thm:fullKK-S2S2}
Let $C_1,C_2$ be exteriors of arbitrary tori in the Kanenobu-Kazama $1$-handle family $F(k,\ell,m)$.  With the one-null boundary identifications above,
\[
\pi_1\bigl(Z_K^{\mathrm{full}}(C_1,C_2)\bigr)=1,
\qquad e=4,\qquad \sigma=0,
\]
and
\[
Q_{Z_K^{\mathrm{full}}(C_1,C_2)}\cong H.
\]
Consequently
\[
Z_K^{\mathrm{full}}(C_1,C_2)\approx \mathbb S^2\times\mathbb S^2.
\]
\end{theorem}
\begin{proof}
Let $W=Y_K\setminus\operatorname{Int}(\nu R_y\cup\nu R_d)$.  Since $v_1=v_2=1$ in the exterior groups, the gluing first imposes $\rho_y=\rho_d=1$.  Hence the image of $\pi_1(W)$ factors through the original filling quotient $\pi_1(Y_K)$.  The classes $\mu_y$ and $\mu_d$ are push-offs of the meridian of $\Sigma_2$.  In the presentation of $\pi_1(Y_K)$ this meridian is represented by $[x,b]$, and the relation $[x,b]=1$ is part of the presentation.  Thus $\mu_y=\mu_d=1$ after the original filling.

The relations $\mu_y=\mu_1$ and $\mu_d=\mu_2$ therefore kill the geometric meridians of the two inserted torus exteriors.  By normal meridional generation, both groups $\pi_1(C_1)$ and $\pi_1(C_2)$ become trivial.  In particular $u_1=u_2=1$, and hence the remaining boundary relations impose $y=d=1$ in the central group.  The displayed presentation of $\pi_1(Y_K)$ then gives
\[
\pi_1(Y_K)/\nn y,d\NN=1.
\]
This proves simple connectivity.  Notice that no assumption is made on whether the inserted complement groups are cyclic, finite-by-$\mathbb Z$, or nonabelian; only the one-null peripheral normal form is used.

The Euler characteristic and signature are unchanged from the doubly-null case: each deleted $\mathbb T^2\times\mathbb D^2$ has $e=\sigma=0$, each inserted torus exterior has $e=2$ and $\sigma=0$, and $Y_K$ has $e=\sigma=0$.  Thus $e=4$ and $\sigma=0$.  The same genus-two surface $\Sigma_2$ and parallel torus fiber $F'$ can be chosen away from the two rim-torus neighborhoods, so they survive with square zero and intersection one.  Since the final manifold is simply connected and has $e=4$, one has $b_2=2$.  The two surviving classes have intersection matrix
\[
\begin{pmatrix}0&1\\1&0\end{pmatrix},
\]
whose determinant is $-1$.  Thus they span a full-rank unimodular sublattice of $H_2$.  Because the ambient intersection lattice is also unimodular, the index of this sublattice is one, so the two classes form an integral basis and the intersection form is $H$.  Freedman's theorem gives the stated homeomorphism type.
\end{proof}

\begin{corollary}\label{cor:fullKK-cases}
Among the resulting inserted exterior groups are:
\begin{enumerate}[label=(\roman*)]
\item the cyclic group $\mathbb Z$, for example $G(1,0,m)$;
\item the nonabelian group $\pi$ of the $6$-twist-spun trefoil, for example $G(0,0,m)$;
\item finite-by-$\mathbb Z$ nonabelian groups $G(k,0,m)$ with $k\ge2$ and $|G'|=k^3$, including the $Q_8$-by-$\mathbb Z$ case $k=2$;
\item the general groups $G(k,\ell,m)$ with $(k,\ell)\ne(0,0)$, whose commutator subgroup is finite by Claim~5 of \cite{KK}.
\end{enumerate}
The closed manifold is nevertheless simply connected because these inserted groups are killed by their geometric meridians in the ordered van Kampen calculation above.
\end{corollary}

\subsection{The homotopy \(\mathbb S^4\) target}

Let \(K\subset\mathbb S^3\) be a knot, let
\[
E_K=\mathbb S^3\setminus\operatorname{Int}\nu K,
\]
and write
\[
\partial(E_K\times\mathbb S^1)\cong\mathbb T^3.
\]
Use the ordered peripheral basis
\[
(m_K,\lambda_K,x),
\]
where \(m_K\) and \(\lambda_K\) are the meridian and preferred longitude of \(K\), and \(x\) is the extra \(\mathbb S^1\)-factor.

Let \((C;\mu,\alpha,\beta)\) be a doubly-null marked torus exterior and choose an orientation-reversing boundary diffeomorphism with
\[
m_K\longmapsto\alpha,\qquad
x\longmapsto\beta,\qquad
\lambda_K\longmapsto\mu.
\]
Denote the resulting closed manifold by
\[
\Sigma_C(K)=(E_K\times\mathbb S^1)\cup_{\Phi}C.
\]

\begin{theorem}
For every knot \(K\subset\mathbb S^3\),
\[
\pi_1(\Sigma_C(K))=1,\qquad
e(\Sigma_C(K))=2,\qquad
\sigma(\Sigma_C(K))=0.
\]
Hence \(\Sigma_C(K)\) is a smooth homotopy \(\mathbb S^4\).
\end{theorem}

\begin{proof}
Let \(G_K=\pi_1(E_K)\). Since \(\alpha=\beta=1\) in \(\pi_1(C)\), Seifert-van Kampen gives
\[
\pi_1(\Sigma_C(K))
\cong
\frac{(G_K\times\langle x\rangle)*\pi_1(C)}
{\langle\!\langle
m_K,\ x,\ \lambda_K\mu^{-1}
\rangle\!\rangle}.
\]
The knot group \(G_K\) is normally generated by the meridian \(m_K\), so the relation \(m_K=1\) kills \(G_K\). In particular \(\lambda_K=1\). Hence the relation \(\lambda_K=\mu\) gives \(\mu=1\), and the meridian \(\mu\) normally generates \(\pi_1(C)\). Therefore \(\pi_1(C)\) also dies. Finally \(x=1\), and
\[
\pi_1(\Sigma_C(K))=1.
\]

Since \(e(E_K)=0\) and \(e(\mathbb S^1)=0\),
\[
e(E_K\times\mathbb S^1)=0.
\]
The common boundary \(\mathbb T^3\) has Euler characteristic zero, while every torus exterior in \(\mathbb S^4\) has \(e(C)=2\). Thus
\[
e(\Sigma_C(K))=2.
\]
Both pieces have signature zero, so Novikov additivity gives
\[
\sigma(\Sigma_C(K))=0.
\]
A closed simply connected \(4\)-manifold with Euler characteristic \(2\) has \(H_2=0\). Hence \(\Sigma_C(K)\) has the integral homology of \(\mathbb S^4\).  Choose an embedded $4$-ball in $\Sigma_C(K)$ and collapse the complement of its interior to a point.  After identifying the quotient with $\mathbb S^4$, this gives a degree-one map
\[
f:\Sigma_C(K)\longrightarrow \mathbb S^4.
\]
Since both manifolds have the integral homology of $\mathbb S^4$, this degree-one map induces an isomorphism on integral homology.  Since both spaces are simply connected CW-complexes, the homological form of Whitehead's theorem implies that \(f\) is a homotopy equivalence.
\end{proof}

\begin{corollary}
If \(K\) is the unknot and \(\Phi\) is the geometric marking coming from the original decomposition
\[
\mathbb S^4=C\cup(\mathbb T^2\times\mathbb D^2),
\]
then
\[
\Sigma_C(K)\cong\mathbb S^4.
\]
\end{corollary}

\begin{proof}
For the unknot,
\[
E_K\times\mathbb S^1\cong \mathbb T^2\times\mathbb D^2.
\]
By hypothesis, \(\Phi\) is the geometric marking obtained from the original decomposition
\[
\mathbb S^4=C\cup_{\partial}(\mathbb T^2\times\mathbb D^2).
\]
Thus, after the above identification of \(E_K\times\mathbb S^1\) with the tubular neighborhood, the defining gluing for \(\Sigma_C(K)\) is precisely the original boundary attachment.  Therefore the resulting closed manifold is diffeomorphic to \(\mathbb S^4\).
\end{proof}

\subsection{The doubly-null subfamily: the $\mathbb S^2\times\mathbb S^2$ target}

Inside the trefoil block $Y_K$, let
\[
 R_y=y\times\mu_{\Sigma_2},\qquad R_d=d\times\mu_{\Sigma_2}
\]
be the two disjoint rim tori placed at different collar levels, as in \cite{AkhPre}.  Remove their neighborhoods and glue two doubly-null marked exteriors $C_1,C_2$ by
\[
 y\mapsto\alpha_1,\quad \rho_y\mapsto\beta_1,\quad \mu_y\mapsto\mu_1,
\]
\[
 d\mapsto\alpha_2,\quad \rho_d\mapsto\beta_2,\quad \mu_d\mapsto\mu_2.
\]
Call the result $Z_K(C_1,C_2)$.

\begin{theorem}\label{thm:S2S2}
For arbitrary doubly-null marked torus exteriors $C_1,C_2$,
\[
 \pi_1(Z_K(C_1,C_2))=1,
 \qquad e(Z_K(C_1,C_2))=4,
 \qquad \sigma(Z_K(C_1,C_2))=0.
\]
Moreover the surviving genus-two surface $\Sigma_2$ and parallel torus fiber $F'$ satisfy
\[
 \Sigma_2^2=(F')^2=0,\qquad \Sigma_2\cdot F'=1,
\]
so
\[
 Q_{Z_K(C_1,C_2)}\cong H
\]
and
\[
 Z_K(C_1,C_2)\approx \mathbb S^2\times\mathbb S^2.
\]
\end{theorem}
\begin{proof}
Let $W=Y_K\setminus\Int(\nu R_y\cup\nu R_d)$ and set $A=\pi_1(W)$.  Before any simplification, van Kampen gives the closed group as the quotient of
\[
 A*\pi_1(C_1)*\pi_1(C_2)
\]
by the six boundary identifications
\[
 y=\alpha_1,\quad \rho_y=\beta_1,\quad \mu_y=\mu_1,
 \qquad
 d=\alpha_2,\quad \rho_d=\beta_2,\quad \mu_d=\mu_2.
\]
Since $\alpha_i=\beta_i=1$, this first imposes
$y=\rho_y=d=\rho_d=1$.  Filling the two deleted rim-torus neighborhoods identifies
$A/\nn\rho_y,\rho_d\NN$ with $\pi_1(Y_K)$; hence the image of $A$ after these four relations is a quotient of
\[
 \pi_1(Y_K)/\nn y,d\NN=1.
\]
Thus the entire central group dies.  In particular $\mu_y=\mu_d=1$, and the remaining boundary relations give $\mu_1=\mu_2=1$.  Each $\mu_i$ normally generates $\pi_1(C_i)$, so both exterior groups then die.  Therefore
\[
 \pi_1(Z_K(C_1,C_2))=1.
\]
For nonabelian exterior groups, the meridians \(\mu_i\) are killed only after the central group has become trivial; normal meridional generation then kills the exterior groups.  Every torus exterior in $\mathbb S^4$ has $e=2$ and $\sigma=0$, hence replacing two $\mathbb T^2\times\mathbb D^2$ neighborhoods changes the Euler characteristic of $Y_K$ from $0$ to $4$ and leaves the signature zero.  The representatives $\Sigma_2$ and $F'$ are chosen away from the replacement regions exactly as in \cite{AkhPre}, so they survive with hyperbolic intersection matrix.  Since simple connectivity and $e=4$ imply $b_2=2$, they form an integral basis.  Freedman's classification of simply connected topological $4$-manifolds gives the homeomorphism type \cite{Freedman}.
\end{proof}

\begin{corollary}\label{cor:KK-S2S2}
For every $k_1,k_2\ge2$,
\[
 Z_K(C_{F_{k_1}},C_{F_{k_2}})\approx \mathbb S^2\times\mathbb S^2,
\]
and the two inserted exteriors have nonabelian groups $G_{k_i}\cong G_{k_i}'\rtimes\mathbb Z$ whose commutator subgroups have orders $k_1^3$ and $k_2^3$, respectively.  These exterior groups are not trivial; they are killed by the boundary relations in the closed-manifold van Kampen calculation.
\end{corollary}

\subsection{The doubly-null subfamily: the $\#_3(\mathbb S^2\times\mathbb S^2)$ target}
The identity-double argument below uses two null boundary slopes and is therefore stated only for doubly-null exteriors.  In particular, the explicit type-zero subfamily of Theorem~\ref{thm:KK-family} applies.  We do not claim here that the same $3H$ construction works for every nonzero cyclic type in the full Kanenobu-Kazama family.

Let
\[
Y_K^\circ=Y_K\setminus\operatorname{int}\nu\Sigma_2
\]
and let $X_K^{\mathrm{id}}$ be the identity double of two copies of $Y_K^\circ$, with the identity on the $\Sigma_2$ factor and reversal of the normal circle.  As in \cite{AkhPre}, choose the two disjoint rim tori $R_y,R_d$ in one side of the gluing neck and place them at distinct collar levels.  Their boundary bases are again
\[
(y,\mu_y,\rho_y),\qquad (d,\mu_d,\rho_d).
\]
Replacing $R_y,R_d$ by doubly-null marked exteriors $C_1,C_2$ with the same boundary assignments as above gives a closed manifold denoted
\[
X_K^{\mathrm{id}}(C_1,C_2).
\]

\begin{theorem}\label{thm:3H}
For arbitrary doubly-null marked torus exteriors $C_1,C_2$,
\[
\pi_1\bigl(X_K^{\mathrm{id}}(C_1,C_2)\bigr)=1,
\qquad e\bigl(X_K^{\mathrm{id}}(C_1,C_2)\bigr)=8,
\qquad \sigma\bigl(X_K^{\mathrm{id}}(C_1,C_2)\bigr)=0.
\]
Moreover
\[
Q_{X_K^{\mathrm{id}}(C_1,C_2)}\cong 3H,
\]
and therefore
\[
X_K^{\mathrm{id}}(C_1,C_2)\approx \#_3(\mathbb S^2\times\mathbb S^2).
\]
\end{theorem}
\begin{proof}
Write
\[
G=\pi_1(Y_K^\circ).
\]
The identity double is the amalgam of two copies of $G$ over
\(\pi_1(\Sigma_2\times\mathbb S^1)\).  We record the part of the complement
calculation from \cite[Lemmas~4.6-4.7]{AkhPre} that is needed here.  In the
first copy the four standard surface generators map to
\[
a^{-1}b,\qquad b^{-1}aba^{-1},\qquad d,\qquad y,
\]
and a meridian of \(\Sigma_2\) is represented by \([x,b]\).  The complement
presentation contains
\[
aba=bab,\qquad dxd^{-1}=xb,\qquad dbd^{-1}=x^{-1}.
\]
Its additional generators \(g_i\) lie in the normal subgroup generated by
\([x,b]\); the remaining extra relator becomes \([x,a]=1\) after those
generators are killed.  The second copy has the corresponding primed
generators.

Now quotient by the normal closure of \(d\) and \(y\).  The product gluing
also gives \(d'=y'=1\).  In the first copy, setting \(d=1\) in
\(dxd^{-1}=xb\) gives \(b=1\).  The relation \(dbd^{-1}=x^{-1}\) then gives
\(x=1\), and the braid relation gives \(a=1\).  Hence \([x,b]=1\), so all
the additional generators \(g_i\) vanish as well.  The same argument applies
to the primed copy.  Thus each vertex group becomes trivial after imposing
\(d=y=1\).

Now perform the two torus-exterior replacements.  Since the two null boundary slopes of $C_1,C_2$ impose $y=\rho_y=d=\rho_d=1$, the image of the central identity-double group is trivial by the preceding calculation.  Consequently the classes $\mu_y,\mu_d$ are trivial.  They are identified with the geometric meridians $\mu_1,\mu_2$ of the inserted exteriors, so $\mu_1=\mu_2=1$.  Each meridian normally generates its exterior group, and hence both nonabelian exterior groups are killed.  Therefore
\[
\pi_1\bigl(X_K^{\mathrm{id}}(C_1,C_2)\bigr)=1.
\]

The identity double before replacement has Euler characteristic $4$ and signature $0$; see \cite{AkhPre}.  Removing two copies of $\mathbb T^2\times\mathbb D^2$ changes neither invariant, and each inserted torus exterior has Euler characteristic $2$ and signature $0$.  Hence the final manifold has $e=8$ and $\sigma=0$.

We spell out the intersection-form argument because it is independent of the nonabelian group calculation and because the location of the replacement tori matters.  Put
\[
\gamma_1=a^{-1}b,\qquad \gamma_2=b^{-1}aba^{-1}.
\]
With the standard genus-two generators of
\cite[Lemma~4.6]{AkhPre}, the ordered curves
\[
(\gamma_1,\gamma_2,d,y)
\]
form a symplectic basis of \(H_1(\Sigma_2;\mathbb Z)\), with
\[
\gamma_1\cdot\gamma_2=1,\qquad d\cdot y=1,
\]
and all other pairings zero.  For clarity, the kernel used below can be
read directly from the abelianization of the complement presentation.
In additive notation the braid relation gives \(a=b\); the relation
\(dxd^{-1}=xb\) then gives \(b=0\), hence \(a=0\), and
\(dbd^{-1}=x^{-1}\) gives \(x=0\).  The meridian \([x,b]\) and the
additional generators from \cite[Lemma~4.7]{AkhPre} vanish in the
abelianization.  No relation remains on \(d\) or \(y\).  Therefore
\[
H_1(Y_K^\circ;\mathbb Z)
   \cong \mathbb Z\langle d\rangle\oplus\mathbb Z\langle y\rangle .
\]
Under the inclusion of the boundary surface, \(\gamma_1,\gamma_2\) map
to zero while \(d,y\) map to these two generators.  Consequently
\[
\ker\{H_1(\Sigma_2;\mathbb Z)\longrightarrow
H_1(Y_K^\circ;\mathbb Z)\}
=\mathbb Z\langle\gamma_1,\gamma_2\rangle.
\]
For $i=1,2$, let $R_i=\gamma_i\times\mu_{\Sigma_2}$ be the corresponding rim torus in the neck, choosing distinct collar levels.  Since $\gamma_i$ is null-homologous in $Y_K^\circ$, choose a properly embedded oriented surface $P_i\subset Y_K^\circ$ with boundary $\gamma_i$.  Doubling \(P_i\) across the identity gluing gives a closed surface \(V_i\) in \(X_K^{\mathrm{id}}\).  With the orientations chosen so that
\[
\gamma_1\cdot\gamma_2=1,
\]
the neck calculation gives
\[
R_1^2=R_2^2=R_1\cdot R_2=0,\qquad
R_1\cdot V_2=1,\qquad
R_2\cdot V_1=-1,
\]
and
\[
R_1\cdot V_1=R_2\cdot V_2=0.
\]
The self-intersection of a doubled relative surface is twice its relative normal Euler number.  To see the parity directly, use the product framing along the common boundary curve in the gluing neck.  The two copies of $P_i$ have the same relative normal Euler number with respect to this framing, and the identity gluing identifies their boundary framings, so the two relative Euler numbers add in the double.  Hence
\[
V_1^2=2m_1,\qquad V_2^2=2m_2
\]
for some \(m_1,m_2\in\mathbb Z\).  Write
\[
V_1\cdot V_2=n.
\]
Define
\[
B_1=V_2-m_2R_1+nR_2,\qquad
B_2=-V_1-m_1R_2.
\]
A direct calculation gives
\[
R_i\cdot B_j=\delta_{ij},\qquad
B_1^2=B_2^2=B_1\cdot B_2=0.
\]
Thus
\[
\langle R_1,B_1,R_2,B_2\rangle\cong2H.
\]

There is a further hyperbolic pair.  Let $S$ be a parallel copy of the gluing surface $\Sigma_2$.  In each half choose a parallel torus fiber $F_i'$ meeting $\Sigma_2$ once.  Removing a small disk at that intersection and gluing the two punctured tori across the neck gives an embedded genus-two surface
\[
T=(F_1'\setminus \mathbb D^2)\cup(F_2'\setminus \mathbb D^2)
\]
with
\[
S^2=T^2=0,\qquad S\cdot T=1.
\]
Hence \(\langle S,T\rangle\cong H\).  We now verify that this
lattice is unaffected by the two replacements.  Choose the intersection
points \(F_i'\cap\Sigma_2\) away from the four curves
\(\gamma_1,\gamma_2,d,y\).  Place \(R_1,R_2\) and the neck portions of
their dual surfaces at collar levels different from those used for
\[
R_d=d\times\mu_{\Sigma_2},\qquad
R_y=y\times\mu_{\Sigma_2}.
\]
The symplectic basis may be represented so that
\(\gamma_1,\gamma_2\) are disjoint from \(d,y\).  The relative surfaces
\(P_i\) can therefore be chosen with product collars disjoint from
\(\nu R_d\cup\nu R_y\); tubing them to parallel rim tori in the
construction of \(B_i\) can also be carried out at the chosen disjoint
neck levels.  The surfaces \(S\) and \(T\) are chosen using the
intersection points above and are likewise disjoint from the replacement
regions.  Thus the representatives \(S,T,R_i,B_i\) all lie in the
complement of \(\nu R_d\cup\nu R_y\).  This is the geometric
disjointness used in \cite[Proposition~4.9]{AkhPre}.  Therefore the full
lattice
\[
H\oplus2H=3H
\]
survives both replacements.

For completeness, the numerical rank agrees with this lattice.  The unreplaced identity double has $e=4$ and $\sigma=0$; each torus-exterior replacement raises $e$ by $2$ and leaves the signature unchanged, so the final manifold has $e=8$ and $\sigma=0$.  Since it is simply connected, $b_1=b_3=0$ and therefore
\[
b_2=e-2=6.
\]
The displayed $3H$ is a unimodular rank-six sublattice of the full intersection lattice.  Since both lattices have rank six, it has finite index, say $q$.  For a full-rank sublattice the determinants satisfy
\[
|\det Q_{3H}|=q^2|\det Q_X|.
\]
Poincare duality makes the ambient intersection form unimodular, while $3H$ is also unimodular, so both determinants have absolute value one.  Hence $q=1$, and the displayed $3H$ is the whole intersection form.  Since $3H$ is even and the manifold is smooth, Freedman's classification gives \cite{Freedman}
\[
X_K^{\mathrm{id}}(C_1,C_2)\approx\#_3(\mathbb S^2\times\mathbb S^2).
\]
\end{proof}

\begin{corollary}\label{cor:KK-3H}
For every $k_1,k_2\ge2$,
\[
X_K^{\mathrm{id}}(C_{F_{k_1}},C_{F_{k_2}})\approx \#_3(\mathbb S^2\times\mathbb S^2).
\]
The two inserted torus exteriors have nonabelian fundamental groups whose commutator subgroups have orders $k_1^3$ and $k_2^3$, respectively, but both groups are killed in the closed-manifold van Kampen calculation above.
\end{corollary}

\section{Swap gluings of one-null torus exteriors}\label{sec:swap}

The constructions above attach torus exteriors to larger four-manifold
blocks. We now consider a direct gluing of two torus exteriors. The effect
on the fundamental group depends on the peripheral marking. A
meridian-null swap kills both exterior groups, whereas a
meridian-preserving gluing of doubly-null exteriors retains them as an
amalgam over the meridian subgroup.

\begin{definition}\label{def:one-null-marked}
A \emph{one-null marked torus exterior} is a torus exterior \(C_T\) together
with a primitive boundary basis
\[
(\mu,u,v)
\]
such that \(\mu\) is the geometric meridian and
\[
v=1\qquad\text{in }\pi_1(C_T).
\]
No condition is imposed on the image of \(u\).
\end{definition}

Let
\[
(C_i;\mu_i,u_i,v_i),\qquad i=1,2,
\]
be one-null marked torus exteriors. Define
\[
\phi_{\mathrm{sw}}:\partial C_1\longrightarrow\partial C_2
\]
on the marked bases by
\[
\mu_1\longmapsto v_2,\qquad
u_1\longmapsto u_2,\qquad
v_1\longmapsto\mu_2.
\]
The corresponding matrix is
\[
A_{\mathrm{sw}}=
\begin{pmatrix}
0&0&1\\
0&1&0\\
1&0&0
\end{pmatrix},
\qquad
\det A_{\mathrm{sw}}=-1.
\]
Thus Lemma~\ref{lem:T3gluing} realizes this assignment by an
orientation-reversing diffeomorphism. Put
\[
X_{\mathrm{sw}}(C_1,C_2)
=
C_1\cup_{\phi_{\mathrm{sw}}}C_2.
\]

\begin{theorem}\label{thm:one-null-swap}
For any two one-null marked torus exteriors,
\[
\pi_1\bigl(X_{\mathrm{sw}}(C_1,C_2)\bigr)=1,
\qquad
e\bigl(X_{\mathrm{sw}}(C_1,C_2)\bigr)=4,
\qquad
\sigma\bigl(X_{\mathrm{sw}}(C_1,C_2)\bigr)=0.
\]
Consequently
\[
H_1(X_{\mathrm{sw}};\mathbb Z)
=
H_3(X_{\mathrm{sw}};\mathbb Z)=0,
\qquad
H_2(X_{\mathrm{sw}};\mathbb Z)\cong\mathbb Z^2.
\]
Its intersection form is either
\[
H
\qquad\text{or}\qquad
\langle1\rangle\oplus\langle-1\rangle.
\]
Accordingly,
\[
X_{\mathrm{sw}}(C_1,C_2)
\approx
\mathbb S^2\times\mathbb S^2
\quad\text{or}\quad
\mathbb{CP}^2\#\overline{\mathbb{CP}}{}^{\,2}.
\]
\end{theorem}

\begin{proof}
Write \(G_i=\pi_1(C_i)\). Seifert-van Kampen gives
\[
\pi_1(X_{\mathrm{sw}})
\cong
\frac{G_1*G_2}
{\left\langle\!\left\langle
\mu_1=v_2,\;
u_1=u_2,\;
v_1=\mu_2
\right\rangle\!\right\rangle}.
\]
Since \(v_1=1\) in \(G_1\) and \(v_2=1\) in \(G_2\), the first and third
relations imply
\[
\mu_1=\mu_2=1.
\]
By Lemma~\ref{lem:normal-meridian}, each \(\mu_i\) normally generates
\(G_i\). Hence both factors \(G_1,G_2\) become trivial. The middle
relation then has no surviving content, and
\[
\pi_1(X_{\mathrm{sw}})=1.
\]

For a torus exterior in \(\mathbb S^4\), Section~\ref{sec:general} gives
\[
e(C_i)=2,\qquad \sigma(C_i)=0.
\]
Since \(e(\mathbb T^3)=0\),
\[
e(X_{\mathrm{sw}})
=e(C_1)+e(C_2)-e(\mathbb T^3)=4.
\]
Novikov additivity gives
\[
\sigma(X_{\mathrm{sw}})=0.
\]
The manifold is closed and simply connected, so \(b_1=b_3=0\). Therefore
\[
4=e(X_{\mathrm{sw}})=2+b_2,
\]
and \(b_2=2\). Poincare duality and the universal coefficient theorem give
a unimodular integral intersection form of rank two and signature zero.
Since the rank is two and the signature is zero, the form is indefinite.  The indefinite unimodular integral forms of rank two and signature zero are the even hyperbolic form
\[
H=
\begin{pmatrix}
0&1\\
1&0
\end{pmatrix}
\]
and the odd form
\[
\langle1\rangle\oplus\langle-1\rangle.
\]
The manifold is smooth, so its Kirby-Siebenmann invariant is zero.  Freedman's classification therefore identifies the even case with $\mathbb S^2\times\mathbb S^2$ and the odd case with $\mathbb{CP}^2\#\overline{\mathbb{CP}}{}^{\,2}$, giving the stated homeomorphism alternatives.
\end{proof}

\begin{remark}\label{rem:swap-parity}
Theorem~\ref{thm:one-null-swap} leaves open the parity of the intersection
form. Thus the swap construction gives a manifold homeomorphic to either
\[
\mathbb S^2\times\mathbb S^2
\quad\text{or}\quad
\mathbb{CP}^2\#\overline{\mathbb{CP}}{}^{\,2}.
\]
Handlebody descriptions of these constructions will be studied separately.
\end{remark}

\subsection{The full Kanenobu-Kazama family}

Lemma~\ref{lem:cyclic-type-one-null} supplies a one-null marking for every
torus in the full Kanenobu-Kazama \(1\)-handle family. Thus the swap
construction is not restricted to the type-zero subfamily.

\begin{corollary}\label{cor:fullKK-swap}
Let \(C_1,C_2\) be exteriors of arbitrary tori in the full
Kanenobu-Kazama \(1\)-handle family. There are primitive boundary markings
for which
\[
\pi_1\bigl(X_{\mathrm{sw}}(C_1,C_2)\bigr)=1,
\qquad e=4,\qquad \sigma=0.
\]
The resulting manifold is homeomorphic to either
\[
\mathbb S^2\times\mathbb S^2
\quad\text{or}\quad
\mathbb{CP}^2\#\overline{\mathbb{CP}}{}^{\,2}.
\]
\end{corollary}

\begin{proof}
Choose the primitive bases \((\mu_i,u_i,v_i)\) from
Lemma~\ref{lem:cyclic-type-one-null}. They satisfy \(v_i=1\) in the
corresponding exterior groups, so Theorem~\ref{thm:one-null-swap} applies.
\end{proof}

The complement groups used in this corollary need not be cyclic. For
example, the subfamily \(G(k,0,m)\), \(k\ge2\), satisfies
\[
|G'(k,0,m)|=k^3,
\]
and hence has nonabelian complement group. The swap kills these groups
because each meridian is identified with the null slope of the opposite
exterior. Thus the conclusion depends on the marked peripheral map, not on
abelianity of the complement group.

\subsection{Meridian-preserving gluings}

For doubly-null exteriors there is another natural boundary identification.
Let
\[
(C_i;\mu_i,\alpha_i,\beta_i),\qquad i=1,2,
\]
satisfy
\[
\alpha_i=\beta_i=1
\qquad\text{in }G_i=\pi_1(C_i).
\]
Choose the orientation-reversing map
\[
\mu_1\longmapsto\mu_2,\qquad
\alpha_1\longmapsto\alpha_2,\qquad
\beta_1\longmapsto\beta_2^{-1}.
\]
The induced matrix has determinant \(-1\).

\begin{proposition}\label{prop:meridian-amalgam}
For this meridian-preserving gluing,
\[
\pi_1(C_1\cup_\phi C_2)
\cong
G_1*_{\langle\mu\rangle}G_2,
\]
where the two infinite cyclic meridian subgroups are identified. In
particular, for the type-zero Kanenobu-Kazama exteriors,
\[
\pi_1(C_{F_{k_1}}\cup_\phi C_{F_{k_2}})
\cong
G_{k_1}*_{\mathbb Z}G_{k_2}.
\]
\end{proposition}

\begin{proof}
Van Kampen gives
\[
\frac{G_1*G_2}
{\left\langle\!\left\langle
\mu_1=\mu_2,\;
\alpha_1=\alpha_2,\;
\beta_1=\beta_2^{-1}
\right\rangle\!\right\rangle}.
\]
The last two relations are already trivial because
\(\alpha_i=\beta_i=1\). Thus the only effective relation identifies the
two meridians. A geometric meridian has infinite order because its
homology class generates
\[
H_1(C_i;\mathbb Z)\cong\mathbb Z.
\]
Hence the identified subgroups are both infinite cyclic, and the displayed
quotient is the claimed amalgamated product.
\end{proof}

\begin{corollary}\label{cor:meridian-amalgam-nonabelian}
If \(k_1,k_2\ge2\), then
\[
G_{k_1}*_{\mathbb Z}G_{k_2}
\]
is nonabelian. Both factors inject into the amalgam, so the finite
commutator subgroups \(G_{k_i}'\), of orders \(k_i^3\), survive as
subgroups.
\end{corollary}

\begin{proof}
The normal-form theorem for free products with amalgamation
\cite[Chapter~IV]{LyndonSchupp} gives injective maps of \(G_{k_1}\) and
\(G_{k_2}\) into \(G_{k_1}*_{\mathbb Z}G_{k_2}\). Since each \(G_{k_i}\)
is nonabelian, the amalgam is nonabelian, and its embedded factors contain
the stated commutator subgroups.
\end{proof}

\begin{remark}
The two gluings use the same exterior pieces but have opposite effects on
the fundamental group. The swap identifies each meridian with a null
slope and absorbs both groups. The meridian-preserving map leaves the null
relations unchanged and joins the two groups along their common meridian.
This is another instance in which the marked peripheral system, rather
than the abstract exterior group alone, controls the closed-manifold
fundamental group.
\end{remark}

\section{Twisted peripheral gluings and controlled fundamental groups}\label{sec:twisted-pi1}

The previous sections use boundary identifications chosen so that the inserted exterior groups are completely absorbed.  A different marking retains a prescribed quotient of one exterior group.  Since the point of this section is the fundamental group, we keep the group calculation explicit and do not identify different quotients unless the required action has been determined.

Retain the rim tori $R_d,R_y\subset Y_K$.  First perform a replacement along $R_d$ using a doubly-null exterior in the same way as in the simply connected construction, so that the class $d$ is killed.  From the presentation of $\pi_1(Y_K)$, imposing $d=1$ gives
\[
 b=1,\qquad x=1,\qquad a=1,
\]
while $y$ survives.  Hence
\[
 \pi_1(Y_K)/\nn d\NN\cong \mathbb Z\langle y\rangle .
\]
Moreover the meridian of $\Sigma_2$ is represented by $[x,b]$, and therefore becomes trivial after imposing $d=1$.

Let $(C;\mu,\alpha,\beta)$ be a doubly-null marked torus exterior.  On $\partial\nu R_y$ start with the basis $(y,\mu_y,\rho_y)$ and replace it by
\[
 \bigl(y,\,y^n\mu_y,\,\rho_y\bigr),\qquad n\in\mathbb Z.
\]
In additive notation on $H_1(\mathbb T^3)$ this change of basis is represented by
\[
\begin{pmatrix}
1&n&0\\
0&1&0\\
0&0&1
\end{pmatrix},
\]
so it is unimodular.  Glue by
\[
 y\longmapsto \mu,\qquad
 y^n\mu_y\longmapsto \alpha,\qquad
 \rho_y\longmapsto \beta.
\]
Since $\alpha=\beta=1$, van Kampen imposes
\[
 \rho_y=1,\qquad y^n\mu_y=1,\qquad y=\mu.
\]
After the first replacement has killed $d$, the class $\mu_y=[x,b]$ is already trivial.  Therefore $y^n=1$, and hence $\mu^n=1$ in the inserted exterior group.

\begin{theorem}\label{thm:meridional-quotient}
Let $C$ be a doubly-null marked torus exterior, let $G=\pi_1(C)$, and let $M_n(C)$ be the closed manifold obtained by the two-step construction above.  Then
\[
 \mathord{\ \pi_1(M_n(C))\cong G/\nn\mu^n\NN .\ }
\]
For $n=0$ no additional meridional relation is imposed.
\end{theorem}
\begin{proof}
After the first replacement, the image of the central group is infinite cyclic generated by $y$, and $\mu_y=1$.  The second gluing identifies $y$ with the exterior meridian $\mu$ and imposes $y^n=1$.  No other relation is introduced in $G$.  Thus the final group is precisely $G/\nn\mu^n\NN$.
\end{proof}

For the Kanenobu-Kazama examples the meridional quotient has the following normal form.

\subsection{Semidirect-product form of a torus-complement group}

Let $G$ be the group of any torus exterior in $\mathbb S^4$.  Alexander duality gives
\[
G_{\mathrm{ab}}\cong H_1(C;\mathbb Z)\cong\mathbb Z,
\]
and the geometric meridian maps to a generator.  Write $N=G'$.  The map
\[
\mathbb Z\longrightarrow G,\qquad 1\longmapsto\mu,
\]
is a section of the abelianization.  Consequently
\[
\mathord{\ G\cong N\rtimes_\varphi\mathbb Z,\ }
\]
where $\varphi(g)=\mu g\mu^{-1}$ is meridional conjugation.  This observation does not assert that $N$ is finite; finiteness is an additional feature of the Kanenobu-Kazama subfamilies considered below.

The next lemma records exactly what imposing $\mu^n=1$ does.

\begin{lemma}\label{lem:semidirect-quotient}
Let
\[
G=N\rtimes_\varphi\langle\mu\rangle,
\qquad \langle\mu\rangle\cong\mathbb Z,
\]
and let $n\ge1$.  Put
\[
N_n=N/\nn \varphi^n(g)g^{-1}\mid g\in N\NN.
\]
Then $\varphi$ descends to an automorphism $\bar\varphi$ of $N_n$ satisfying $\bar\varphi^n=1$, and
\[
\mathord{
G/\nn\mu^n\NN\cong N_n\rtimes_{\bar\varphi}\mathbb Z_n.
}
\]
In particular, if $\varphi^n=1$ already on $N$, then
\[
G/\nn\mu^n\NN\cong N\rtimes_\varphi\mathbb Z_n.
\]
\end{lemma}
\begin{proof}
In the quotient by $\mu^n$, conjugation by $\mu^n$ is trivial.  Hence every $g\in N$ satisfies
\[
g=\mu^n g\mu^{-n}=\varphi^n(g),
\]
so the indicated relations are necessary.  Conversely, after imposing them on $N$, the action of $\mu$ has order dividing $n$, and the standard semidirect-product presentation becomes
\[
\left\langle N_n,\mu\ \middle|\ \mu g\mu^{-1}=\bar\varphi(g),\ \mu^n=1\right\rangle,
\]
which is $N_n\rtimes_{\bar\varphi}\mathbb Z_n$.
\end{proof}

Thus the finite normal subgroup is controlled by $N=G'$, while the quotient depends on the meridional action on $N$.

\subsection{The Kanenobu-Kazama presentations after twisted gluing}

For the $1$-handle $h_w$ on the $6$-twist-spun trefoil, Section~\ref{sec:KK} gives
\[
G_w=
\left\langle x,y\ \middle|\
xyx=yxy,\ [x^6,y]=1,\
[x,x^jwx^{-j}]=1\ (0\le j\le5)
\right\rangle,
\]
with meridian $x$.  Therefore the twisted construction has the completely explicit presentation
\[
\mathord{
\pi_1(M_n(C_w))=
\left\langle x,y\ \middle|\
xyx=yxy,\ [x^6,y]=1,\
[x,x^jwx^{-j}]=1\ (0\le j\le5),\ x^n=1
\right\rangle .
}
\]
This presentation is valid without determining the structure of the resulting finite quotient.

The relation $[x^6,y]=1$ shows that $x^6$ commutes with $y$, and it plainly commutes with $x$.  Hence $x^6$ is central in the original $6$-twist-spun group and remains central in every quotient $G_w$.  It follows that meridional conjugation satisfies
\[
\varphi^6=1\quad\text{on }G_w'.
\]
Consequently, whenever $n=6r$, Lemma~\ref{lem:semidirect-quotient} has $N_n=N$ and gives
\[
\mathord{
\pi_1(M_{6r}(C_w))\cong G_w'\rtimes_{\varphi}\mathbb Z_{6r}.
}
\]
This is the reason for restricting the clean finite-group formulas below to exponents divisible by six.  For a general $n$, the correct group is the quotient described in Lemma~\ref{lem:semidirect-quotient}; additional relations may be imposed on $G_w'$.

\subsection{Cyclic examples}

For the $k=1$ member of the family $G(k,0,m)$, Kanenobu-Kazama's Claim~6 \cite{KK} gives
\[
|G'(1,0,m)|=1.
\]
Therefore the full complement group is infinite cyclic:
\[
G(1,0,m)\cong\mathbb Z\langle\mu\rangle.
\]
The meridional quotient formula gives
\[
\mathord{
\pi_1(M_n(C_{F_1}))\cong\mathbb Z_n\qquad(n\ge1),
}
\]
and
\[
\pi_1(M_0(C_{F_1}))\cong\mathbb Z.
\]
Thus the construction realizes the trivial group, every finite cyclic group, and the infinite cyclic group by varying only the boundary marking.

\begin{remark}
For this one-meridian quotient construction, an abelian resulting group must be cyclic.  Indeed,
\[
\left(G/\nn\mu^n\NN\right)_{\mathrm{ab}}
\cong \mathbb Z_n.
\]
Thus a noncyclic abelian fundamental group cannot occur from this particular meridional quotient.  Obtaining, for example, $\mathbb Z_p\oplus\mathbb Z_q$ requires a construction in which two independent abelian generators survive the central block.
\end{remark}

\subsection{The quaternion-by-cyclic family}

For $k=2$ and the doubly-null choice $m=1$, Kanenobu-Kazama state explicitly \cite[Remark after Claim~7]{KK} that
\[
G_2'=G'(2,0,1)\cong Q_8.
\]
Hence
\[
G_2\cong Q_8\rtimes_{\varphi_2}\mathbb Z.
\]
The relation $[x^6,y]=1$ implies $\varphi_2^6=1$.  Therefore for every $r\ge1$,
\[
\mathord{
\pi_1(M_{6r}(C_{F_2}))
\cong Q_8\rtimes_{\varphi_2}\mathbb Z_{6r}.
}
\]
Its order is
\[
\mathord{\left|\pi_1(M_{6r}(C_{F_2}))\right|=8\cdot6r=48r.}
\]
It is nonabelian because it contains the normal subgroup $Q_8$.

Kanenobu-Kazama identify the commutator subgroup with $Q_8$.  We do not determine the automorphism $\varphi_2$; for the order and nonabelianity statements it is enough that meridional conjugation has order dividing six.

For exponents not divisible by six one still has an exact description:
\[
\pi_1(M_n(C_{F_2}))
\cong (Q_8)_n\rtimes\mathbb Z_n,
\]
where
\[
(Q_8)_n=
Q_8/\nn\varphi_2^n(q)q^{-1}\mid q\in Q_8\NN.
\]
Depending on $n$ and on the action $\varphi_2$, this quotient can be smaller than $Q_8$.  We make no stronger identification without computing $\varphi_2$ explicitly.

\subsection{Odd $k$: Heisenberg commutator groups}

Suppose $k>1$ is odd and choose the doubly-null parameter $m=k$.  Kanenobu-Kazama's Claim~6 \cite{KK} gives $|G_k'|=k^3$ and exhibits a surjection of $G_k'$ onto the upper-unitriangular group
\[
H_k=\operatorname{UT}_3(\mathbb Z_k)
=
\left\{
\begin{pmatrix}
1&r&s\\
0&1&t\\
0&0&1
\end{pmatrix}
\; ;\; r,s,t\in\mathbb Z_k
\right\}.
\]
Since $|H_k|=k^3=|G_k'|$, this map is an isomorphism.  Hence
\[
\mathord{G_k'\cong H_k\qquad(k>1\text{ odd}).}
\]
Equivalently, $H_k$ has a presentation
\[
H_k=\left\langle A,B,C\ \middle|\
A^k=B^k=C^k=1,\ C=[A,B],\ [A,C]=[B,C]=1
\right\rangle .
\]
For $k>1$ this group is nonabelian; its center and commutator subgroup are both generated by $C$.

The full torus-complement group is therefore
\[
G_k\cong H_k\rtimes_{\varphi_k}\mathbb Z,
\qquad \varphi_k^6=1.
\]
For $n=6r$ we obtain the explicit finite nonabelian family
\[
\mathord{
\pi_1(M_{6r}(C_{F_k}))
\cong H_k\rtimes_{\varphi_k}\mathbb Z_{6r},
\qquad k>1\text{ odd},
}
\]
with
\[
\mathord{
\left|\pi_1(M_{6r}(C_{F_k}))\right|=6r\,k^3.
}
\]
For example,
\[
k=3,\quad r=1
\]
gives
\[
\pi_1(M_6(C_{F_3}))
\cong \operatorname{UT}_3(\mathbb Z_3)\rtimes_{\varphi_3}\mathbb Z_6,
\]
a nonabelian group of order
\[
27\cdot6=162.
\]
We do not need the precise automorphism $\varphi_k$ for the group order or nonabelianity.

\subsection{Even $k$: the finite commutator subgroup from Claim~7}

Let $k=2s\ge2$ and choose the doubly-null parameter $m=s$.  Kanenobu-Kazama's Claim~7 \cite{KK} gives a presentation of $G_k'$ as an extension of a finite group $N_k$ by a cyclic group of order $2s$, and proves
\[
|G_k'|=k^3.
\]
For $k=2$ this specializes to $Q_8$.  For general even $k$ we use their extension description rather than replacing it by an unproved familiar group.

Thus
\[
G_k\cong G_k'\rtimes_{\varphi_k}\mathbb Z,
\qquad \varphi_k^6=1,
\]
and for every $r\ge1$,
\[
\mathord{
\pi_1(M_{6r}(C_{F_k}))
\cong G_k'\rtimes_{\varphi_k}\mathbb Z_{6r},
\qquad k\text{ even},
}
\]
with
\[
\mathord{
\left|\pi_1(M_{6r}(C_{F_k}))\right|=6r\,k^3.
}
\]
The group is nonabelian for $k\ge2$.  Indeed, the quotient map
$G_k\to G_k/\nn\mu^{6r}\NN$ sends the commutator subgroup of $G_k$ onto the commutator subgroup of the quotient.  Since $\varphi_k^{6r}=1$, the normal closure of $\mu^{6r}$ meets the embedded factor $G_k'$ trivially in the semidirect-product normal form.  Hence the derived subgroup of the finite quotient is isomorphic to $G_k'$, which has order $k^3>1$.  We do not assign a standard name to $G_k'$ unless Kanenobu-Kazama identify it explicitly.

\subsection{A uniform statement}

The preceding calculations can be summarized as follows.

\begin{theorem}\label{thm:finite-KK}
Let $C_{F_k}$ be one of the doubly-null exteriors from Theorem~\ref{thm:KK-family}.  Write
\[
G_k=\pi_1(C_{F_k})=G_k'\rtimes_{\varphi_k}\mathbb Z,
\]
where the infinite cyclic factor is generated by the meridian.  Then
\[
\varphi_k^6=1.
\]
For every $r\ge1$,
\[
\pi_1(M_{6r}(C_{F_k}))
\cong G_k'\rtimes_{\varphi_k}\mathbb Z_{6r}
\]
and
\[
\left|\pi_1(M_{6r}(C_{F_k}))\right|=6r\,k^3.
\]
For every $k\ge2$ the resulting group is finite and nonabelian.  The separate $k=1$ cyclic case discussed above gives $\mathbb Z_{6r}$.
\end{theorem}
\begin{proof}
The splitting of $G_k\to\mathbb Z$ is supplied by the meridian.  Since $x^6$ is central in the $6$-twist-spun trefoil group, its image is central in $G_k$, and therefore $\varphi_k^6=1$.  Apply Theorem~\ref{thm:meridional-quotient} and Lemma~\ref{lem:semidirect-quotient} with $n=6r$.  The order formula follows from $|G_k'|=k^3$.
\end{proof}

\begin{remark}
For a general exponent $n$, Theorem~\ref{thm:meridional-quotient} remains valid, but the finite normal subgroup is
\[
(G_k')_n=G_k'/\nn\varphi_k^n(g)g^{-1}\mid g\in G_k'\NN,
\]
which can be a proper quotient of $G_k'$.  Thus the formula $G_k'\rtimes\mathbb Z_n$ should not be asserted for arbitrary $n$ without computing the meridional action.
\end{remark}

\section*{Acknowledgments}

The first author thanks Çağrı Karakurt and Cliff Taubes for helpful discussions.  This work was motivated by the first author's earlier preprint \cite{AkhPre}, where the corresponding replacement construction was studied for torus exteriors with infinite cyclic fundamental group. LLM tool were used for assistance for checking grammar and language editing. All mathematical ideas, arguments, and results in this paper are due to the authors, who takes full responsibility for the content.

\section*{Author information}

\noindent
\textbf{Anar Akhmedov}\\
School of Mathematics, University of Minnesota, Minneapolis, Minnesota, USA.\\
Visiting Scholar, Department of Mathematics, Harvard University, Cambridge, Massachusetts, USA.\\
Email: \texttt{akhmedov@math.umn.edu}

\medskip

\noindent
\textbf{Azer Akhmedov}\\
Department of Mathematics, North Dakota State University, Fargo, North Dakota, USA.\\
Email: \texttt{azer.akhmedov@ndsu.edu}

\end{document}